\documentclass[
final
]{dmtcs-episciences}

\usepackage{eucal}
\usepackage{microtype}
\usepackage{nicefrac}

\usepackage{amsmath,amsthm}
\usepackage{mathtools}
\usepackage{hyperref}
\usepackage[ocgcolorlinks]{ocgx2}
\usepackage[nameinlink]{cleveref}

\usepackage{todonotes}

\usepackage[shortlabels, inline]{enumitem}
\setlist{itemsep=0pt,topsep=2pt}

\usepackage{booktabs}

\usepackage[natbib, useprefix, maxbibnames=9, sortcites=true]{biblatex}
\usepackage{comment}
\usepackage{ifthen}
\usepackage{xspace}

\usepackage{xcolor}
\colorlet{darkgreen}{green!50!black}
\colorlet{dgreen}{darkgreen}
\colorlet{medgray}{gray!75}
\colorlet{lightgray}{gray!30}
\definecolor{linkcol}{rgb}{0,0,0.4} 
\definecolor{citecol}{rgb}{0.5,0,0} 

\usepackage{tikz}
\usetikzlibrary{arrows.meta,shapes,positioning,calc,patterns,backgrounds,decorations,intersections}
\pgfdeclarelayer{background2}
\pgfdeclarelayer{background}
\pgfdeclarelayer{foreground}
\pgfsetlayers{background2,background,main,foreground}

\tikzstyle{bold}=[draw, line width=2pt]
\tikzstyle{optional}=[dashed]
\tikzstyle{path}=[decorate, decoration={snake, amplitude=.6mm}]

\tikzstyle{small}=[inner sep=2pt]
\tikzstyle{tiny}=[inner sep=1.7pt]
\tikzstyle{textnode}=[inner sep=0pt]

\tikzstyle{triangle}=[draw, regular polygon, regular polygon sides=3]
\tikzstyle{vertex}=[circle, draw, fill=white]
\tikzstyle{reti}=[vertex, fill=black]
\tikzstyle{leaf}=[vertex, rectangle]
\tikzstyle{leaf2}=[vertex, regular polygon, regular polygon sides=3]

\tikzstyle{smallvertex}=[vertex, small]
\tikzstyle{smallleaf}=[leaf, inner sep=3.3pt]
\tikzstyle{smallleaf2}=[leaf2, inner sep=1.7pt]
\tikzstyle{smalltriangle}=[triangle, inner sep=1.5pt]
\tikzstyle{smallreti}=[reti, small]

\tikzstyle{tinyvertex}=[vertex, tiny]

\tikzstyle{toroot}=[smallvertex, fill=white]
\tikzstyle{normal}=[smallvertex, fill=black]

\tikzstyle{match}=[edge,line width=3pt]
\tikzstyle{bgmatch}=[match, opacity=.5]

\tikzstyle{matching}=[edge,line width=3pt]
\tikzstyle{solution}=[gray!60, line width=5pt]
\tikzstyle{arc}=[draw,arrows={-Latex[length=6pt]}]
\tikzstyle{boldarc}=[draw, bold, arrows={-Latex[length=10pt]}]
\tikzstyle{revarc}=[draw, arrows={Latex[length=6pt]-}]
\tikzstyle{boldrevarc}=[draw, bold, arrows={Latex[length=10pt]-}]
\tikzstyle{HRL}=[gray, bgmatch]

\tikzstyle{reverseclip}=[insert path={(current page.north east) --
  (current page.south east) --
  (current page.south west) --
  (current page.north west) --
  (current page.north east)}
]

\tikzset{
  >={Triangle[width=0pt 2 0,length=0pt 3 0]},
  edge/.style={draw=black,very thick},
  elided path/.style={draw=black,thick,loosely dotted},
  subtree/.style={draw=black,thin,fill=black!20},
  region/.style={draw=black!50,fill=black!10,densely dotted},
  transform/.style={draw=black,->,line width=6pt},
  cherry/.style={rounded corners=9pt},
  green cherry/.style={cherry,fill=green!75!black,opacity=0.1},
  red cherry/.style={cherry,fill=red,opacity=0.1},
  blue cherry/.style={cherry,fill=blue,opacity=0.1}
}

\newenvironment{claimproof}{{\noindent\textit{Proof. }}}{\hfill$\lrcorner$} 

\newcommand{\dist}[2][]{\ensuremath{\operatorname{dist}}\ifx\relax#1\relax\else\ensuremath{_{#1}}\fi\ensuremath{(#2)}}
\newcommand{\adist}[2][]{\ensuremath{\overline{\operatorname{dist}}}\ifx\relax#1\relax\else\ensuremath{_{#1}}\fi\ensuremath{(#2)}}
\renewcommand{\deg}[2][]{\ensuremath{\operatorname{deg}}\ifx\relax#1\relax\else\ensuremath{_{#1}}\fi\ensuremath{(#2)}}
\newcommand{\pred}[2][]{\ensuremath{\operatorname{N}^\text{in}}\ifx\relax#1\relax\else\ensuremath{_{#1}}\fi\ensuremath{(#2)}}
\newcommand{\inE}[2][]{\ensuremath{\operatorname{E}^\text{in}}\ifx\relax#1\relax\else\ensuremath{_{#1}}\fi\ensuremath{(#2)}}
\newcommand{\indeg}[2][]{\ensuremath{\operatorname{indeg}}\ifx\relax#1\relax\else\ensuremath{_{#1}}\fi\ensuremath{(#2)}}
\newcommand{\suc}[2][]{\ensuremath{\operatorname{N}^\text{out}}\ifx\relax#1\relax\else\ensuremath{_{#1}}\fi\ensuremath{(#2)}}
\newcommand{\outE}[2][]{\ensuremath{\operatorname{E}^\text{out}}\ifx\relax#1\relax\else\ensuremath{_{#1}}\fi\ensuremath{(#2)}}
\newcommand{\outdeg}[2][]{\ensuremath{\operatorname{outdeg}}\ifx\relax#1\relax\else\ensuremath{_{#1}}\fi\ensuremath{(#2)}}

\newcommand{\nh}[2][]{\ensuremath{\operatorname{N}}\ifx\relax#1\relax\else\ensuremath{_{#1}}\fi\ensuremath{(#2)}}
\newcommand{\inc}[2][]{\ensuremath{\operatorname{inc}}\ifx\relax#1\relax\else\ensuremath{_{#1}}\fi\ensuremath{(#2)}}
\newcommand{\LCA}[2][]{\ensuremath{\operatorname{LCA}}\ifx\relax#1\relax\else\ensuremath{_{#1}}\fi\ensuremath{(#2)}}

\newcommand{\widthop}[3]{%
  \ensuremath{%
    \ifx\relax#1\relax%
      \operatorname{#3}_{#2}%
    \else%
      \ifx\relax#2\relax%
        \operatorname{#3}\left({#1}\right)%
      \else
        \operatorname{#3}_{#2}\left({#1}\right)%
      \fi
    \fi%
  }%
}

\newcommand{\LW}[3][]{\ensuremath{\operatorname{C}}\ifx\relax#1\relax\else\ensuremath{_{#1}}\fi^{#2}\ensuremath{(#3)}}
\newcommand{\cost}[2][]{\ensuremath{\operatorname{cost}}\ifx\relax#1\relax\else\ensuremath{_{#1}}\fi\ensuremath{(#2)}}
\newcommand{\Rcost}[2][]{\ensuremath{\operatorname{cost}^R}\ifx\relax#1\relax\else\ensuremath{_{#1}}\fi\ensuremath{(#2)}}

\newcommand{\PROB}[2]{\expandafter\newcommand\csname #1\endcsname{\textsc{#2}\xspace}}
\PROB{RA}{Register Allocation}
\PROB{HP}{Hamiltonian Path}
\PROB{HC}{Hamiltonian Cycle}
\PROB{DHP}{Directed Hamiltonian Path}
\PROB{DHC}{Directed Hamiltonian Cycle}
\PROB{ST}{Steiner Tree}
\PROB{TSAT}{3-Satisfyability}
\PROB{WTSAT}{Weighted 2-Satisfyability}
\PROB{MWTS}{Monotone Weighted 2-Satisfyability}
\PROB{VC}{Vertex Cover}
\PROB{HS}{Hitting Set}
\PROB{IS}{Independent Set}
\PROB{OCT}{Odd Cycle Transversal}
\PROB{TSP}{Traveling Salesman}
\PROB{HRL}{Horizontal Red Line Traversal}
\PROB{NAESAT}{Not-All-Equal 3-SAT}
\PROB{TC}{Tree Containment}
\PROB{CC}{Cluster Containment}
\PROB{STC}{Soft Tree Containment}
\PROB{TUIS}{2-Union Independent Set}
\PROB{LCS}{Longest Common Subsequence}
\PROB{SCSe}{Shortest Common Supersequence}
\PROB{SCSt}{Shortest Common Superstring}
\PROB{MCSP}{Minimum Common String Partition}
\PROB{EFSP}{Equality-Free String Partition}
\PROB{CS}{Closest String}
\PROB{CSu}{Closest Substring}

\PROB{SumCSP}{Sum CSP}

\ifx\qedhere\undefined
  \newcommand\qedhere{\hfill\qed}
\fi

\newcommand\nil[1][.7em]{%
	\makebox[#1]{%
		\kern.07em
		\vrule height.6ex
		\hrulefill
		\vrule height.6ex
		\kern.07em
	}
}

\let\leaves\lbls

\ifx\backref\undefined
\else
\renewcommand*{\backref}[1]{}
\renewcommand*{\backrefalt}[4]{%
\ifcase #1 %
(Not cited.)%
\or
(Cited on page~#2.)%
\else
(Cited on pages~#2.)%
\fi}

\fi

\makeatletter
\def\NAT@spacechar{~}
\makeatother

\newcommand{\mybox}[2]{
  \begin{tikzpicture}
    \node[minimum width=\linewidth-0.4pt, draw, rounded corners, text width=\linewidth-12pt] (a){#2};
    \node[fill=white, xshift=1em, anchor=west] at (a.north west) {#1};
  \end{tikzpicture}
}

\newcommand{\myboxprobdef}[5]{
  \label{#5}
  \mybox{%
    \ifthenelse{\equal{#3}{}}{}{{#3}\ifthenelse{\equal{#4}{}}{}{ ({#4})}}
  }{%
    \begin{compactdesc}
      \item [Input:] {#1}
      \item [Question:] {#2}
    \end{compactdesc}
  }
}

\newcommand{\probdef}[5]{
\hbox{\vbox{
\begin{quote}
  \label{#5}
  \ifthenelse{\equal{#3}{}}{}{{#3}\ifthenelse{\equal{#4}{}}{}{ ({#4})}}
  \vspace{-1ex}
  \begin{compactdesc}
    \item [Input:] {#1}
    \item [Question:] {#2}
  \end{compactdesc}
\end{quote}
}}
}

\newcommand{\taskprobdef}[5]{
\hbox{\vbox{
\begin{quote}
  \label{#5}
  \ifthenelse{\equal{#3}{}}{}{{#3}\ifthenelse{\equal{#4}{}}{}{ ({#4})}}
  \begin{compactdesc}
    \item [Input:] {#1}
    \item [Task:] {#2}
  \end{compactdesc}
\end{quote}
}}
}

\newcommand{\paraproblem}[6]{
\hbox{\vbox{
\begin{quote}
  \label{#6}
  \ifthenelse{\equal{#4}{}}{}{{#4}\ifthenelse{\equal{#5}{}}{}{ ({#5})}}
  \begin{compactdesc}
    \item [Input:] {#1}
    \item [Question:] {#2}
    \item [Parameter:] {#3}
  \end{compactdesc}
\end{quote}
}}
}

\RequirePackage{hyperref}
\hypersetup
{
bookmarksopen=false,
pdftoolbar=false, 
pdfmenubar=true, 
pdfhighlight=/O, 
ocgcolorlinks=true,
pdfpagemode=UseNone, 
pdfpagelayout=SinglePage, 
pdffitwindow=true, 
linkcolor=linkcol, 
citecolor=citecol, 
urlcolor=linkcol 
}

\RequirePackage{cleveref}

\theoremstyle{plain}

\ifx\newdefinition\undefined
  \let\newdefinition\newtheorem
\fi
\ifx\definition\undefined
  \newdefinition{definition}{Definition}

\fi

\newtheorem{thm}{Theorem}
\newtheorem{lem}[thm]{Lemma}
\newtheorem{prop}[thm]{Proposition}
\newtheorem{cor}[thm]{Corollary}
\newtheorem{obs}[thm]{Observation}
\theoremstyle{definition}

\newtheorem{claim}{Claim}

\newdefinition{construction}{Construction}

\crefname{section}{Section}{Sections}
\crefname{table}{Table}{Tables}
\crefname{figure}{Figure}{Figures}
\crefname{lemma}{Lemma}{Lemmas}
\crefname{corollary}{Corollary}{Corollaries}
\crefname{definition}{Definition}{Definitions}
\crefname{transformation}{Transformation}{Transformations}
\crefname{proposition}{Proposition}{Propositions}
\crefname{claim}{Claim}{Claims}
\crefname{observation}{Observation}{Observations}
\crefname{construction}{Construction}{Constructions}
\crefname{sem}{Semantics}{Semantics}
\crefname{rrule}{Rule}{Rules}
\crefname{trule}{Tree Rule}{Tree Rules}
\crefname{prule}{Path Rule}{Path Rules}
\crefname{asm}{Assumption}{Assumptions}
\crefname{lem}{Lemma}{Lemmas}
\crefname{thm}{Theorem}{Theorems}
\crefname{cor}{Corollary}{Corollaries}
\crefname{obs}{Observation}{Observations}
\crefname{prop}{Proposition}{Propositions}
\crefname{enumi}{}{}
\crefname{equation}{}{}

\crefformat{item}{(#2#1#3)}
\crefrangeformat{item}{(#3#1#4--#5#2#6)}
\crefmultiformat{item}{(#2#1#3}{, #2#1#3)}{, #2#1#3}{, #2#1#3)}

\ifx\bibsection\undefined\else
\makeatletter
\renewcommand\bibsection%
{
  \section*{\refname
    \@mkboth{\MakeUppercase{\refname}}{\MakeUppercase{\refname}}}
}
\makeatother
\fi

\usepackage{subcaption}
\usepackage{tikz}
\usetikzlibrary{calc,intersections,through,fit}

\tikzset{
  edge/.style = {draw=black,thick,line cap=round,line join=round},
  edge0/.style={edge,densely dashed},
  pair/.style={draw=black!40,fill=black!15,inner sep=1pt,rounded corners=8pt},
  network/.style = {edge,thin,fill=black!20!white},
  cutout/.style = {draw=white,line width=2.4pt},
  tree/.style={draw,thick},
  t1/.style={tree,black},
  t2/.style={tree,blue!90!black},
  t3/.style={tree,red!80!black},
  t4/.style={tree,cyan!70!black},
  t5/.style={tree,orange!70!black},
  t6/.style={tree,cyan!90},
  sedge/.style={fill=black!10},
  sedge boundary/.style={very thin,draw=black!20},
  strand/.style={fill=black!30},
  strand boundary/.style={very thin,draw=black!40},
  sbundle/.style={fill=black!20,draw=black!80,densely dotted},
  svertex/.style={very thin,draw=black!50,fill=black!40},
  embed/.style={line width=2pt}
}

\let\lg\log
\let\epsilon\varepsilon

\newcommand\M{\mathcal{M}}
\newcommand\N{\mathcal{N}}
\renewcommand\S{\mathcal{S}}
\newcommand\T{\mathcal{T}}
\newcommand\NN{\mathbb{N}}

\let\netlab\Lambda
\let\tuplab\Gamma
\newcommand\unrooted[1]{#1^{\text{u}}}
\newcommand{\netlabnr}[1][n,r]{\ensuremath{\netlab_{#1}}}
\newcommand{\unetlabnr}[1][n,r]{\ensuremath{\unrooted{\netlab}_{#1}}}
\newcommand{\tuplabnr}[1][n,r]{\ensuremath{\tuplab_{#1}}}
\newcommand{\utuplabnr}[1][n,r]{\ensuremath{\unrooted{\tuplab}_{#1}}}
\newcommand{\NLEnet}[1][]{\ensuremath{(#1 N, #1 \lambda_E)}}
\newcommand{\uNLEnet}[1][]{\ensuremath{(#1 \unrooted{N}, #1 \lambda_E)}}
\newcommand{\NLEtup}[1][]{\ensuremath{\langle #1 N, #1 \lambda_E\rangle}}
\newcommand{\uNLEtup}[1][]{\ensuremath{\langle #1 \unrooted{N}, #1 \lambda_E\rangle}}

\publicationdata{vol. 28:2}{2026}{13}{10.46298/dmtcs.16446}{2025-09-02; 2025-09-02; 2026-02-06}{2026-02-08}

\title[When Many Trees Go to War]{When Many Trees Go to War:\\On Sets of Phylogenetic Trees With Almost No Common Structure\thanks{Most of this work was carried out while both authors were at LaBRI, Université Bordeaux, France.}}

\author{Mathias Weller\affiliationmark{1,2} \and Norbert Zeh\affiliationmark{3}\thanks{The work of the second author was supported by NSERC Discovery Grant number RGPIN-2025-06235.}}
\affiliation{
LIGM, Université Gustave Eiffel, Marne~la~Vallée, France\\
Centre National de Recherche Scientifique, France\\
Faculty of Computer Science, Dalhousie University, Halifax, Canada}

\hypersetup{urlcolor=black}

\begin{document}
\maketitle

\begin{abstract}\noindent
  It is known that any two trees on the same $n$~leaves can be displayed by a
  network with $n-2$ reticulations, and there are two trees that cannot be
  displayed by a network with fewer reticulations.
  But how many reticulations are needed to display multiple trees?  For any set
  of $t$ trees on $n$ leaves, there is a trivial network with $(t - 1)n$
  reticulations that displays them. To do better, we have to exploit common
  structure of the trees to embed non-trivial subtrees of different trees into
  the same part of the network.
  In this paper, we show that, for $t \in o\left(\sqrt{\lg n}\right)$, there is a set
  of $t$ trees with virtually no common structure that could be exploited.
  More precisely, we show that, for any $t\in o\left(\sqrt{\lg n}\right)$, there are $t$~trees
  such that any network displaying them has $(t-1)n - o(n)$~reticulations.  For
  $t \in o(\lg n)$, we obtain a slightly weaker bound.  We also prove that
  already for $t = c\lg n$, for any constant $c > 0$, there is a set of $t$
  trees that cannot be displayed by a network with $o(n \lg n)$ reticulations,
  matching up to constant factors the known upper bound of $O(n \lg n)$
  reticulations sufficient to display \emph{all} trees with $n$ leaves.
These results are based on simple counting arguments and extend to unrooted
networks and trees. \end{abstract}

\section{Introduction}\label{sec:intro}

%
%
Phylogenetics is the study of evolutionary relationships among a set of taxa,
typically represented using tree-like structures that depict patterns of
descent from common ancestors.
Traditionally, these relationships have been modelled with phylogenetic trees,
which are rooted or unrooted trees whose leaves correspond to present-day taxa
and whose internal nodes represent ancestral lineages~\cite{SS03}.
However, due to processes such as hybridization, horizontal gene transfer, and
recombination, tree models often fail to capture the full complexity of
evolutionary histories~\cite{Doo99,Mal05,DR09,BMR+09}. This has led to the
development of phylogenetic networks, which are graphical models that
generalize trees by allowing for ``reticulate'' (non-tree-like)
events~\cite{HRS10,Mor11}.

Both rooted and unrooted trees and networks are considered in practice,
depending on the data and the biological questions at hand. Rooted structures
incorporate ancestral/temporal directionality, which is important for inferring
evolutionary pathways, while unrooted structures are often used when such
directionality is unknown or when analyzing symmetric distance
data~\cite{SS03,HS12}.

%
A central concept in the analysis of phylogenetic networks is that of a network
``displaying'' a tree: informally, a network displays a tree if the tree can be
embedded in the network in a way that respects its topology~\cite{HRS10}.  A
formal definition will be given in \cref{sec:preliminaries}.
This notion is at the heart of computational problems in phylogenetics like
comparing phylogenies or parsimoniously reconstructing networks~\cite{BS07}.
Indeed, the computational problems \textsc{Tree Containment} (decide whether a
given tree is embeddable into a given network)~\cite{KNL+08, ISS10, BS16,
GDZ17, Gun18, Wel18} and \textsc{Hybridization Number} (construct a smallest
network displaying all given
trees)~\citep{BSS05,BGMS05,bordewich2007computing,van2011two,kelk2012cycle}
have been studied extensively.
%
%

%
There is a trivial network that displays any set of $t$~trees without
exploiting any common structure (see \cref{fig:trivial-network}).
This network is obtained by simply
\begin{enumerate*}[(1)]
  \item taking the disjoint union of the $t$~trees,
  \item adding a tree ``cap'' that joins the roots of the trees under a common root and, finally,
  \item adding reticulations to ``merge'' all copies of each leaf in the different trees.
\end{enumerate*}
This network has $(t - 1)n$~reticulations and conveys absolutely no information
about the similaries between the $t$~trees.
In general, this is far from the minimum number of reticulations necessary to
display a set of trees: if multiple input trees share subtrees with the same
topology, then these subtrees can be embedded into the same part of the
network, which may allow us to save reticulations.  Indeed, Wu and Zhang \cite{WZ25} proved
that, for three trees, $2n - 2 - \Omega(\lg \lg n)$ reticulations always suffice and,
in general, $(t - 1)(n - 2) - \Omega(\min\{t^{\nicefrac32},n^3\})$ reticulations always suffice.
Yet, the gap between these bounds and the number of reticulations in the
trivial network is only $o(tn)$. An interesting question then is whether there
exist sets of trees with little enough common structure to make the trivial
network the best possible up to such lower order terms.

\Citet{BSS05} answered this question in the affirmative for $t = 2$, by showing
that there are two $n$-caterpillars requiring $n-2$~reticulations.

Recently, \citet{vIJW24} proved that even displaying $3$~caterpillars without
restrictions on the type of network may require $(\nicefrac{3}{2} -
o(1))n$~reticulations. This left open the possibility that the number of
reticulations required to display $t$ trees depends sublinearly on $t$, and
they explicity pose the general case of $t$~trees as an open question. This
paper settles this question by showing that the dependence on $t$ is linear
after all, up to $t \in o(\lg n)$.

Another known special case is $t=(2n-3)!!$, that is, how many reticulations are
necessary to display \emph{all} trees on the same~$n$ leaves? \Citet{BS18}
showed that this number is in $\Theta(n\lg n)$ even for tree-based networks,
but the $\Omega(n\lg n)$ lower bound is derived via a simple counting argument
applied to any type of network.
In this paper, we generalize this counting argument to the case of $t$~trees.
We show that for $t \in o(\lg n)$, there is a set of $t$~trees for which the
trivial network (see \Cref{fig:trivial-network}) is asymptotically optimal, in
the sense that there are $t$~trees that cannot be displayed by a network with
fewer than $(t - 1)n - o(tn)$~reticulations. For~$t \in o\left(\sqrt{\lg n}\right)$, we
obtain the slightly stronger result that there are $t$ trees that cannot be
displayed by a network with fewer than $(t - 1)n - o(n)$ reticulations.
Thus, we conclude that the trivial network is indeed \emph{asymptotically}
optimal in the worst case for displaying $t$~trees.
The proof of these two results also shows that, for any constant $c > 0$, there
is a set of $c\lg n$ trees that cannot be displayed by a network with $o(n
\lg n)$ reticulations.  One way to interpret this result is that ``most'' of
the reticulations in the $\Theta(n \lg n)$-reticulation network of \citet{BS18}
that displays all trees on $n$ leaves are caused by a small subset of only
$O(\lg n)$ trees, and adding the remaining exponentially many trees increases
the number of reticulations necessary to display these trees by only
a constant factor.

As the results by \citet{BS18,vIJW24,BSS05} show, the question of how many
reticulations are necessary to display a set of $t$ trees has been the focus of
previous work and is interesting in its own right.  Our interest in this
question stems from its connection to cluster reduction.  Cluster reduction is
a technique that has been verified experimentally to be the most important
technique for computing maximum agreement forests of two trees
efficiently~\cite{LZ17}.  Before this experimental study, \citet{BSS06} proved
that it can also be used to reduce the complexity of finding a network with the
minimum number of reticulations that displays two trees, owing to the fact that,
for two trees, computing such a network is equivalent to computing a particular
type of agreement forest \cite{BS07}.  We say that cluster reduction is
``safe'' for two trees.  But what happens for more than two trees?  This
question is studied in \cite{IJL+25}. One of the main results of that paper is
that cluster reduction fails to be safe for four or more trees, that is, the
use of cluster reduction may result in computing a network with more than the
minimum number of reticulations.  The proof is via a reduction that starts with
a set of trees that require many reticulations to display and constructs from
it a new set of trees for which cluster reduction fails.  This proof relies
directly on the results shown in this paper.

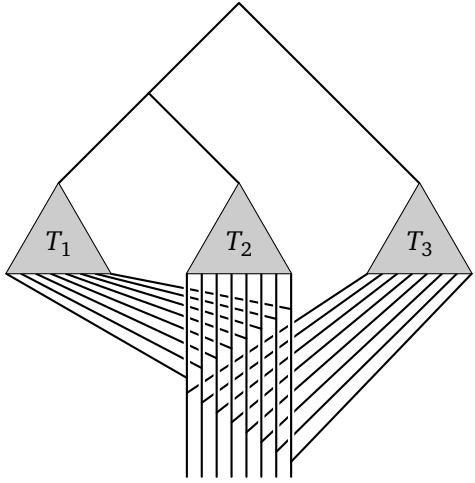
\begin{figure}[t]
  \centering
  \begin{tikzpicture}
    \foreach [count=\p from 1] \x in {a,b,c} {
      \path [network]
        (\p*2.4,0) coordinate (\x1) --
        ++(0:1.4) coordinate (\x8) --
        ++(120:1.4) coordinate (r\x) -- cycle;
      \path ($(\x1)!0.5!(\x8)$) node [anchor=south,yshift=1mm] {$T_{\p}$};
      \path
        (\x1) +(0:0.2pt) coordinate (\x1)
        (\x8) +(180:0.2pt) coordinate (\x8)
        foreach \i in {2,...,7} {
          ($(\x1)!{(\i-1)/7}!(\x8)$) coordinate (\x\i)
        };
    }
    \begin{scope}[overlay]
      \path [name path=a] (a1) -- +(330:6);
      \path [name path=b] (b1) -- +(270:6);
      \path [name intersections={of=a and b}] (intersection-1) coordinate (pa1);
      \path (pa1) ++(270:0.2) coordinate (pc1);
      \path [name path=a] (pc1) -- (c1);
      \path [name path=b] (b8) -- +(270:6);
      \path [name intersections={of=a and b}] (intersection-1) +(90:0.2) coordinate (pa8);
      \path
        (ra) +(330:0.2pt) coordinate (ra)
        (rb) +(210:0.2pt) coordinate (rb)
        (rc) +(210:0.2pt) coordinate (rc);
      \path
        (pa1 -| b8) coordinate (pc8)
        ($(pa8)!2!(pc8)$) ++(270:0.2) coordinate (pc8);
      \path [name path=a] (ra) -- +(45:6);
      \path [name path=b] (rb) -- +(135:6);
      \path [name intersections={of=a and b}] (intersection-1) coordinate (rab);
      \path [name path=b] (rc) -- +(135:6);
      \path [name intersections={of=a and b}] (intersection-1) coordinate (r);
    \end{scope}
    \path
      (pc8) ++(270:0.2) coordinate (l8)
      (b1 |- l8) coordinate (l1)
      foreach \i in {2,...,7} {
        ($(pa1)!{(\i-1)/7}!(pa8)$) coordinate (pa\i)
        ($(pc1)!{(\i-1)/7}!(pc8)$) coordinate (pc\i)
        ($(l1)!{(\i-1)/7}!(l8)$) coordinate (l\i)
      };
    \begin{scope}[on background layer]
      \path [edge]
        (ra) -- (r) -- (rc)
        (rab) -- (rb)
        foreach \i in {1,...,8} {
          (a\i) -- (pa\i)
          (c\i) -- (pc\i)
        };
      \path [cutout]
        foreach \i in {1,...,7} {
          (pa\i) +(90:0.1) -- (b\i)
        }
        foreach \i in {2,...,8} {
          (pa\i) +(270:0.1) -- ([shift=(90:0.1)]pc\i)
        };
      \path [edge]
        foreach \i in {1,...,8} {
          (b\i) -- (l\i)
        };
    \end{scope}
  \end{tikzpicture}
  \caption{The trivial network displaying three trees $T_1$, $T_2$, and $T_3$,
  each with leaf set $[8]$}
  \label{fig:trivial-network}
\end{figure}

\section{Preliminaries}\label{sec:preliminaries}

\paragraph*{Phylogenetic networks.}
An \emph{unrooted phylogenetic network} $N^{\text{u}}$ (\emph{unrooted network}
for short) is a simple connected undirected graph whose degree-one nodes (called \emph{leaves})
have unique labels and whose non-leaf nodes have degree at least~$3$ and are
unlabelled. If $N^{\text{u}}$ does not contain cycles, then we
call~$N^{\text{u}}$ an unrooted phylogenetic tree (\emph{unrooted tree} for
short). If the maximum degree in $N^{\text{u}}$ is at most three, then we call
$N^{\text{u}}$~\emph{binary}.
A \emph{rooted phylogenetic network}~$N$ (\emph{network} for short) is a
simple connected directed acyclic graph with a single source, called the \emph{root}.
Its out-degree-$0$ nodes, called \emph{leaves}, have unique labels and
its non-leaf, non-root nodes are unlabelled and either have
in-degree~$1$ and out-degree at least~$2$ (\emph{tree nodes}) or
out-degree~$1$ and in-degree at least~$2$ (\emph{reticulations}).
If~$N$ has no reticulations, it is called a \emph{rooted phylogenetic tree}
(\emph{tree} for short).
If both the maximum in-degree and the maximum out-degree of~$N$ are at most two,
then~$N$ is called \emph{binary}.
Throughout this paper, all considered (unrooted) networks are binary.
We use~$V(N)$, $E(N)$, and~$\leaves(N)$ to refer to the nodes, edges, and
leaves of an (unrooted) network~$N$, respectively.
If $(u,v)$ is a directed edge of a network~$N$,
then $u$ is called a \emph{parent} of $v$ in $N$,
and $v$ is called a \emph{child} of $u$ in $N$.
Similarly, any in-edge $(u,v)$ of a node $v$ is called a \emph{parent edge} of $v$,
and any out-edge $(v,w)$ of $v$ is called a \emph{child edge} of~$v$.

\paragraph*{Labellings.}
The labels of the leaves of a tree or network are part of its definition.
Whenever we need to refer to the label of a leaf $v$ explicitly, we use the
notation $\lambda_L(v)$. We assume that all networks (rooted and unrooted) have
their leaves labelled with the integers~$1$~to~$n$, where $n$ is the number of
leaves. We use $[n]$ to refer to the set $\{1, \ldots, n\}$.

For counting networks, it will be useful to also label the edges of networks.
If $N$ is an (undirected) network and $\lambda_E$ is a labelling of its edges,
we call the pair $(N,\lambda_E)$ an (undirected) \emph{edge-labelled network}.
When only a subset of the edges is labelled, we extend this to a labelling of
all edges by assigning the special label $\bot$ to unlabelled edges.
Accordingly, we call the edges with label $\bot$ \emph{unlabelled}. Similarly,
it is convenient to extend the leaf labelling $\lambda_L$ of $N$ to a labelling
of all nodes of $N$, by assigning the label $\bot$ to every non-leaf node.

For two mappings $\phi, \psi : A \to B$ with $\bot \in B$, we say that $\psi$
\emph{extends} $\phi$, or~$\psi$ is an \emph{extension} of~$\phi$, if~$\psi(a)
= \phi(a)$ whenever $\phi(a) \ne \bot$.  In other words, $\psi$ may label more
elements than $\phi$ does but~$\phi$ and~$\psi$ agree on the labels of all
elements that $\phi$ does label.  Viewing $\phi$ and $\psi$ as relations, this
implies that $\phi \setminus (A \times \{\bot\}) \subseteq \psi \setminus (A
\times \{\bot\})$. Therefore, we write $\phi \subseteq \psi$ whenever $\psi$ is
an extension of $\phi$.

\paragraph*{Isomorphisms.}
An \emph{isomorphism} $\phi : G_1 \to G_2$ between directed
graphs~$G_1$ and~$G_2$ is a pair of (set) isomorphisms
$\phi_V : V(G_1) \to V(G_2)$ and $\phi_E : E(G_1) \to E(G_2)$ such that,
for every edge $(u,v) \in E(G_1)$,
$\phi_E((u,v)) = (\phi_V(u), \phi_V(v))$.
The definition for undirected graphs is analogous.
An isomorphism between (undirected) edge-labelled networks~$(N, \lambda_E)$
and~$(N', \lambda_E')$ is an (undirected) graph isomorphism~$\phi: N \to N'$
with the added conditions that~$\lambda_L(v) = \lambda_L'(\phi_V(v))$ for all
$v \in V(N)$, and~$\lambda_E(e) = \lambda_E'(\phi_E(e))$ for all $e \in E(N)$.
Throughout this paper, we consider isomorphic networks to be the same network.

\paragraph{Displaying a tree.}
A \emph{subdivision} of a (rooted or unrooted) tree~$T$ is a graph~$T'$
obtained by replacing each edge~$e\in E(T)$ with a path~$P_e$ consisting of
one or more edges.
The paths in $\{P_e \mid e \in E(T)\}$ are internally
node-disjoint, and their internal nodes are unlabelled.
If~$T$ is rooted, then $P_{(u,v)}$ is a directed path from $u$ to $v$, for all $(u,v) \in E(T)$.
Note that such a subdivision $T'$ is a (rooted or unrooted) tree in the
graph-theoretic sense but not according to our definition of a phylogenetic
tree, as it may have nodes with in-degree~$1$ and out-degree~$1$.

A (rooted or unrooted) network $N$ \emph{displays} a (rooted or unrooted)
tree~$T$ if $N$ has a subgraph $T'$ that is (isomorphic to) a subdivision of $T$.
We call $T'$ an \emph{embedding} of $T$ into $N$.
Note that, in general, $T'$~is not unique, that is,
a tree may have multiple embeddings into a network that displays it.

\paragraph*{Useful (in)equalities.}
We will refer to the base-2 logarithm as $\lg$ or $\lg_2$ in this paper.
Further, we rely on the following equalities and inequalities.

\begin{obs}\label{obs:double-fac-eq}
  $\displaystyle (2k - 1)!! = \frac{(2k)!}{2^kk!}$.
\end{obs}
%

\noindent
The relations $\binom{n}{k} = \frac{n!}{(n - k)!k!}$ and $\binom{n}{k} \le 2^n$
immediately imply that

\begin{obs}\label{obs:binom-bound}
  $\displaystyle (a + b)! \le 2^{a + b}a!b!$.
\end{obs}

\begin{lem}\label{lem:not-quite-e}
  $\displaystyle \frac1e < \left(\frac{n}{n + 1}\right)^n \le \frac12$, for all $n \in \mathbb{N}^+$. 
\end{lem}

\begin{proof}
  For $n = 1$, we have $(\nicefrac{n}{n+1})^n = \nicefrac12$.
  For $n\to\infty$, we have
  \[
    \lim_{n \to \infty} (\nicefrac{n}{n+1})^n 
    = \lim_{n \to \infty} (\nicefrac{n}{n+1})^{n+1} \cdot (\nicefrac{n+1}{n}) 
    = \lim_{n \to \infty} (1 - \nicefrac{1}{n+1})^{n+1} \cdot \lim_{n \to \infty} (1 + \nicefrac1n) 
    = \nicefrac 1e \cdot 1.
  \]
  Since the derivative of $(\nicefrac{n}{n+1})^n$ is
  non-positive, for all $n \ge 0$, the lemma follows.
\end{proof}

\begin{lem}\label{lem:fac-bound}
  $\left(\nicefrac{n}{e}\right)^n < n! < \left(\nicefrac{n}{2}\right)^n$, for all $n \ge 6$.
\end{lem}

\begin{proof}
  The first inequality holds for all~$n\in\mathbb{N}^+$ and follows from Stirling's bound.
  The proof of the second inequality is by induction on~$n$.
  It can be verified that the inequality holds for $n=6$.
  For all~$n \ge 6$,
  \begin{align*}
    (n+1)! 
    &= (n+1) \cdot n! 
    \stackrel{\text{IH}}{<} (n+1) \cdot \frac{n^n}{2^n} 
    = (n+1) \cdot \frac{(n+1)^n}{2^n} \cdot \left(\frac{n}{n+1}\right)^n\\
    &= \frac{(n+1)^{n+1}}{2^{n+1}} \cdot 2\left(\frac{n}{n+1}\right)^n
    \stackrel{\text{Lem.~\ref{lem:not-quite-e}}}{\le}
    \frac{(n+1)^{n+1}}{2^{n+1}}.\qedhere
  \end{align*}
\end{proof}

\begin{lem}\label{lem:double-fac-bound}
  $(2n-3)!! > \left(\nicefrac{n}{2}\right)^n$, for all $n \ge 5$.
\end{lem}

\begin{proof}
  The proof is by induction on~$n$.
  It can be verified that the claim holds for $n=5$.
  For all~$n \ge 5$,
  \begin{align*}
    (2(n+1)-3)!! = (2n-1)(2n-3)!!
    \stackrel{\text{IH}}{>} (2n-1)\left(\frac{n}{2}\right)^n
    =\frac{2n-1}{n+1}\left(\frac{n+1}{2}\right)^{n+1}\cdot 2\left(\frac{n}{n+1}\right)^n.
  \end{align*}
  Since $\nicefrac{2n-1}{n+1}\geq \nicefrac{3}{2}$ for all $n\geq 5$ and, by \cref{lem:not-quite-e},
  $(\nicefrac{n}{n+1})^n > \nicefrac 1e$ for all $n \in \mathbb{N}^+$,
  the claim follows.
\end{proof}

\section{Counting Rooted Binary Trees and Networks}

Our proof that some sets of trees cannot be displayed by networks with few
reticulations uses a simple counting argument.
First, we count the number of sets of $t$~trees on $n$~leaves.
Then, given some integer~$r$, we count how many such sets can be displayed by any one network with at most~$r$~reticulations.
If~$r$ is the smallest integer such that each set of $t$~trees on $n$~leaves can be displayed by some network with $r$~reticulations,
then there is a set of $t$~trees on $n$~leaves that cannot be displayed by a network with fewer than $r$~reticulations.
Our goal in this paper is to obtain lower bounds on this number~$r$.
We present the proof for rooted trees and networks in this section and consider the unrooted case in the next section.

\begin{prop}\label{prop:num-networks}
  Let~$n,r\in\NN^+$, and
  let~$\mathcal{N}_{n,r}$ be the set of all binary networks
  with $n$~leaves and $r$~reticulations.
  %
  Then $|\mathcal{N}_{n,r}| \le \frac{(2n + 4r - 3)!!}{r!} \leq n! \cdot (r - 1)! \cdot 2^{2n+6r-3}$.
\end{prop}

\begin{proof}
  A \emph{switching} $\sigma : E(N) \to \{0, \bot\}$ of $N$ is a labelling of
  the edges of $N$ such that every non-root node of $N$ has exactly one
  labelled parent edge. Thus, $N$ has exactly $r$~unlabelled edges, one parent
  edge per reticulation. A~\emph{reticulation labelling} $\lambda_E$ of $N$ is
  an extension of a switching $\sigma$ that bijectively labels all edges left
  unlabelled by $\sigma$ with labels in $[r]$. Note that every switching
  $\sigma$ of $N$ gives rise to $r!$ reticulation labellings corresponding to
  all possible permutations of the $r$ edges of $N$ left unlabelled
  by~$\sigma$.
  A~\emph{reticulation-labelled network} is an edge-labelled network $(N,
  \lambda_E)$, where $N \in \mathcal{N}_{n,r}$ and $\lambda_E$ is a
  reticulation labelling of~$N$.
  Now, for every $N \in \mathcal{N}_{n,r}$,
  choose an arbitrary but fixed switching~$\sigma_{N}$ of $N$, and let
  \[ 
    \netlabnr := \{\NLEnet \mid N \in \mathcal{N}_{n,r}
    \land \lambda_E \text{ is a reticulation labelling of $N$}
    \land \sigma_{N} \subseteq \lambda_E\}
  \]
  where, by convention, two edge-labelled networks are considered equal
  if they are isomorphic.
  Let $\T_m$ be the set of all binary phylogenetic trees with $m$ leaves.
  We prove that
  \begin{align*}
    |\mathcal{N}_{n,r}| \cdot r! \stackrel{\text{(a)}}{=} |\netlabnr| \stackrel{\text{(b)}}{\leq} |\T_{n+2r}|,
  \end{align*}
  implying that $|\mathcal{N}_{n,r}| \le \frac{(2n + 4r - 3)!!}{r!}$,
  as~$|\T_{n+2r}| = (2(n+2r)-3)!!$.
  The bound $\frac{(2n + 4r - 3)!!}{r!} \leq n! \cdot (r - 1)! \cdot 2^{2n+6r-3}$
  is obtained using basic algebra then.

  \begin{figure}
    \hspace{\stretch{1}}%
    \subcaptionbox{}{\begin{tikzpicture}[x=8mm,y=8mm]
        \begin{scope}[overlay]
          \path coordinate (1)
            ++(45:1) coordinate (t8)
            ++(45:1) coordinate (t9)
            ++(315:1) coordinate (t7)
            ++(225:1) coordinate (r2)
            ++(270:1) coordinate (t4)
            ++(315:1) coordinate (r3)
            ++(270:1) coordinate (t2)
            ++(225:1) coordinate (r4)
            ++(270:1) coordinate (3)
            ++(0:1) coordinate (4);
          \path [name path=a] (t4) -- +(225:2);
          \path [name path=b] (r4) -- +(135:2);
          \path [name intersections={of=a and b}] (intersection-1) coordinate (t3) ++(225:1) coordinate (2);
          \path [name path=a] (t2) -- +(315:2);
          \path [name path=b] (4) -- +(45:2);
          \path [name intersections={of=a and b}] (intersection-1) coordinate (t1) ++(315:1) coordinate (r1) ++(270:1) coordinate (5);
          \path [name path=a] (t9) -- +(315:6);
          \path [name path=b] (r3) -- +(45:2);
          \path [name intersections={of=a and b}] (intersection-1) coordinate (t6);
          \path [name path=b] (r1) -- +(45:3);
          \path [name intersections={of=a and b}] (intersection-1) coordinate (t5) ++(315:1) coordinate (6);
        \end{scope}
        \path foreach \i in {1,...,6} {
            (\i) node (\i) {\i}
          };
        \path [edge]
          (1) -- (t9) -- (6)
          (2) -- (t4) -- (r2) -- (t7)
          (3) -- (r4) -- (t2) -- (r1) -- (5)
          (t2) -- (r3) -- (t6)
          (t1) -- (4);
        \path [edge0]
          (t8) -- (r2)
          (t4) -- (r3)
          (t3) -- (r4)
          (t5) -- (r1);
        \path
          (r1) to node [anchor=north west,inner sep=2pt] {1} (t5)
          (r2) to node [anchor=north east,inner sep=2pt] {2} (t8)
          (r3) to node [anchor=north east,inner sep=2pt] {3} (t4)
          (r4) to node [anchor=north east,inner sep=2pt] {4} (t3);
    \end{tikzpicture}}%
    \hspace{\stretch{2}}%
    \subcaptionbox{}{\begin{tikzpicture}[x=8mm,y=8mm]
        \path node (1) {1}
          ++(45:1) coordinate (t1)
          +(315:1) node (9) {9}
          ++(45:2) coordinate (t2)
          ++(315:2) coordinate (t3)
          ++(315:3) coordinate (t7)
          ++(315:3) coordinate (t13)
          +(225:1) node (7) {7}
          +(315:1) node (6) {6}
          (t3) ++(225:2) coordinate (t4)
          +(225:1) node (10) {10}
          ++(315:1) coordinate (t5)
          +(315:1) node (11) {11}
          ++(225:2) coordinate (t6)
          +(225:1) node (2) {2}
          +(315:2) node (13) {13}
          (t7) ++(225:2) coordinate (t8)
          +(225:1) node (12) {12}
          ++(315:1) coordinate (t9)
          ++(315:1) coordinate (t11)
          +(225:1) node (4) {4}
          ++(315:1) coordinate (t12)
          +(225:1) node (5) {5}
          +(315:1) node (8) {8}
          (t9) ++(225:2) coordinate (t10)
          +(225:1) node (14) {14}
          +(315:1) node (3) {3};
        \path [edge]
          (1) -- (t2) -- (6)
          (t3) -- (t4) -- (t5) -- (2)
          (t7) -- (t8) -- (t12) -- (5)
          (t9) -- (t10) -- (3)
          (t11) -- (4)
          ;
        \path [edge0]
          (t1) -- (9)
          (t4) -- (10)
          (t5) -- (11)
          (t6) -- (13)
          (t8) -- (12)
          (t13) -- (7)
          (t10) -- (14)
          (t12) -- (8);
        \node (7') at (7) {\phantom{10}};
        \begin{scope}[on background layer]
          \node [pair,fit={(7') (8)}] {};
          \node [pair,fit={(9) (10)}] {};
          \node [pair,fit={(11) (12)}] {};
          \node [pair,fit={(13) (14)}] {};
        \end{scope}
    \end{tikzpicture}}%
    \hspace{\stretch{1}}
    \caption{%
      (a) A reticulation-labelled network $\NLEnet$ with
      $6$~leaves and $4$~reticulations.
      Edges labelled~$0$ are drawn solid.
      Edges labelled~$1$ through~$4$ are drawn dashed.
      (b) The tree~$\tau(\NLEnet)$.
      Pairs of leaves representing the same reticulation
      of~$\NLEnet$ are shaded.
    }
    \label{fig:network-to-tree}
  \end{figure}
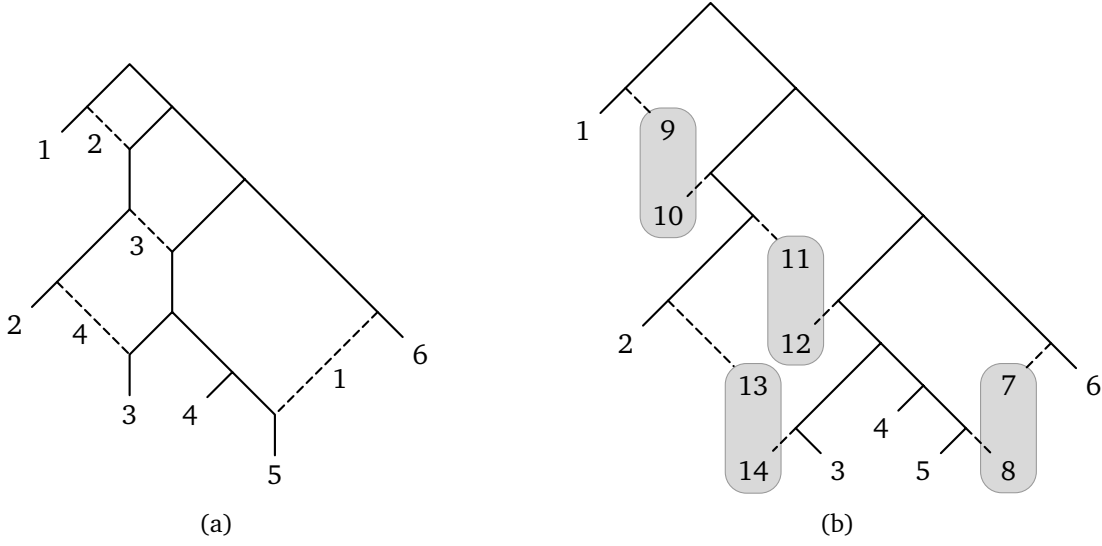

  \medskip\noindent (b):
  To prove (b),
  we construct an injective mapping~$\tau : \netlabnr \to \T_{n+2r}$.
  For each~$\NLEnet \in \netlabnr$,
  we construct~$\tau(\NLEnet)$ as follows:
  We replace every edge~$e = (u,v)$
  such that~$\lambda_E(e) \ne 0$ with two edges $(u,z)$ and $(v,z')$, where
  $z$ and $z'$ are new leaves.  These leaves are assigned labels~$\lambda_L(z)
  := n + 2\lambda_E(e) - 1$ and $\lambda_L(z') := n + 2\lambda_E(e)$.
  Finally, we drop all edge labels. This construction is illustrated in
  \cref{fig:network-to-tree}.  Since $\tau(\NLEnet)$ is a
  binary phylogenetic tree with $n+2r$ leaves labelled $1$~through~$n+2r$, we
  have $\tau(\NLEnet) \in \T_{n+2r}$.
  This mapping~$\tau$ is injective because the network $\NLEnet$ is easily
  recovered from~$\tau(\NLEnet)$ by reversing the above
  construction:
  We label all edges of $\tau(\NLEnet)$ with label~$0$.
  Then, for all~$h \in [r]$, let $u$ and $v$ be the parents of the
  leaves~$z$ and $z'$ with labels $\lambda_L(z) = n + 2h - 1$ and
  $\lambda_L(z') = n + 2h$, respectively.
  We remove $z$ and $z'$ and their parent edges,
  and add a new edge~$(u, v)$ with label $\lambda_E((u, v)) = h$.
  It is easily verified that this construction indeed recovers
  the edge-labelled network~$\NLEnet$ from $\tau(\NLEnet)$.

  \medskip\noindent (a):
  Since we consider isomorphic networks to be the same network in this paper, the
  proof that $|\mathcal{N}_{n,r}| = \frac{|\netlabnr|}{r!}$ needs to
  distinguish between the two interpretations of a
  pair~$\NLEnet$ as simply a pair or as an edge-labelled network.
  If~$\lambda_E \ne \lambda_E'$, then the two
  pairs~$\NLEnet$ and $(N,\lambda_E')$
  are not the same but the two
  edge-labelled networks~$\NLEnet$ and $(N,\lambda_E')$ may be isomorphic and,
  therefore, the same to us.
  To make this distinction explicit, we refer to~$\NLEnet$ as~$\NLEtup$ when
  we view it as simply a pair, and as~$\NLEnet$
  when we view it as an edge-labelled network. Let
  \[
    \tuplabnr := 
    \{\NLEtup \mid
    N \in \mathcal{N}_{n,r}
    \land \lambda_E \text{ is a reticulation labelling of $N$}
    \land \sigma_{N} \subseteq \lambda_E
    \},
  \]
  where two pairs $\NLEtup$ are equal if their components are equal.
  Then $|\mathcal{N}_{n,r}|\cdot r! = |\tuplabnr|$ because,
  as already observed, there are $r!$~reticulation labellings
  extending~$\sigma_{N}$, for each~$N \in \mathcal{N}_{n,r}$.
  Thus, to prove that $|\mathcal{N}_{n,r}|\cdot r! = |\netlabnr|$,
  we need to prove that~$|\netlabnr| = |\tuplabnr|$.
  Since it is clear that $|\netlabnr|\leq |\tuplabnr|$,
  it suffices to show the following:

  \begin{claim}\label{cl:phi-injective}
    If $\NLEnet$ and $\NLEnet[\tilde]$ are isomorphic, then
    $\NLEtup = \NLEtup[\tilde]$.
  \end{claim}
  \begin{claimproof}
    Consider two pairs $\NLEtup, \NLEtup[\tilde] \in \tuplabnr$ such that
    there is an isomorphism~$\phi: \NLEnet \to \NLEnet[\tilde]$.
    Then $\phi$ is also an isomorphism~$\phi: N \to \tilde N$,
    so $N = \tilde N$.
    We show that $\phi$ must be the identity~$1_N$,
    which implies that~$\phi$ maps every edge of~$N$ onto itself and, therefore, that
    $\lambda_E = \tilde\lambda_E$. This in turn implies that $\NLEtup = \NLEtup[\tilde]$.

    Assume for the sake of contradiction that~$\phi \ne 1_N$.
    Since there are no parallel edges in~$N$,
    there is a lowest node~$v \in V(N)$ with~$\phi(v) \ne v$,
    that is, $\phi(w) = w$ holds for all children~$w$ of $v$ in $N$.
    Since~$\phi: N \to N$ is an isomorphism,
    it respects the labels of all leaves in~$N$.
    Therefore, $\phi(x) = x$ for all $x \in \leaves(N)$,
    so~$v$ is not a leaf of~$N$.
    Let~$w$ be a child of~$v$.
    If~$w$ has a single parent edge~$e$, then $\phi(e) = e$, since $\phi(w) = w$ and $\phi$ is an isomorphism.
    If~$w$ has two parent edges~$e_1$ and~$e_2$, then~$\lambda_E(e_1) \ne \lambda_E(e_2)$, since~$\lambda_E$ is a reticulation labelling
    (one of~$\lambda_E(e_1)$ and~$\lambda_E(e_2)$ is zero; the other, non-zero).
    Since~$\phi(w) = w$ and $\phi$~respects edge labels, $\phi(e_1) = e_1$ and~$\phi(e_2) = e_2$.
    This shows that no matter whether $w$ has one or two parents, we have $\phi(e) = e$,
    for every parent edge~$e$ of $w$.
    If we choose $e$ to be the parent edge whose top endpoint is $v$,
    then this implies that $\phi(v) = v$, contradicting our assumption.
    Therefore, $\phi = 1_N$ and $\NLEtup = \NLEtup[\tilde]$.
  \end{claimproof}

  \medskip\noindent
  Having shown (a) and (b), the proposition follows from
  \begin{align*}
    |\mathcal{N}_{n,r}|
    &\stackrel{\text{(a),(b)}}{\le} \frac{|\T_{n+2r}|}{r!} = \frac{(2n + 4r - 3)!!}{r!}\\
    &\stackrel{\text{Obs.~\ref{obs:double-fac-eq}}}{=} \frac{(2n + 4r - 2)!}{2^{n + 2r - 1}(n + 2r - 1)!r!}
    \stackrel{\text{Obs.~\ref{obs:binom-bound}}}{\le} \frac{2^{2n + 4r - 2}((n + 2r - 1)!)^2}{2^{n + 2r - 1}(n + 2r - 1)!r!}
    = \frac{2^{n + 2r - 1}(n + 2r - 1)!}{r!}\\
    &\stackrel{\text{Obs.~\ref{obs:binom-bound}}}{\le} \frac{2^{2n + 4r - 2}n!(2r - 1)!}{r!}
    \stackrel{\text{Obs.~\ref{obs:binom-bound}}}{\le} \frac{2^{2n + 6r - 3}n!r!(r - 1)!}{r!}
    = 2^{2n + 6r - 3}n!(r - 1)!.\qedhere
  \end{align*}
\end{proof}

Since each $n$-leaf network with $r$~reticulations displays at most $2^r$
trees, each such network displays at most $\binom{2^r}{t}\leq 2^{rt}$ sets
of $t$ trees.  Thus, we immediately obtain the following corollary.

\begin{cor}\label{cor:num-net-tree-pairs}
  Let~$n,t,r\in\NN^+$, and
  let~$\mathcal{M}_{n,t,r}$ be the set of all pairs~$(N,\T)$ such that
  $\T$~is a set of $t$~binary trees on $n$~leaves, and
  $N$~is a binary network with $n$~leaves and $r$~reticulations displaying all trees in~$\T$.
  Then $|\mathcal{M}_{n,t,r}| \leq 2^{rt}\cdot n! \cdot (r - 1)! \cdot 2^{2n+6r-3}$.
\end{cor}

\begin{obs}\label{obs:num-trees}
  Let~$n,t\in\NN^+$ with $n\geq 5$, and
  let~$\S_{n,t}$ be the set of all
  sets of $t$ binary trees on the leaf set~$[n]$.
  Then
  \begin{equation*}
    |\S_{n,t}|
    = \binom{(2n-3)!!}{t}
    \geq \left(\frac{(2n - 3)!!}{t}\right)^t
    \stackrel{\text{Lem.~\ref{lem:double-fac-bound}}}{>} \left(\frac{n}{2}\right)^{tn}\cdot t^{-t}.
  \end{equation*}
\end{obs}

\begin{thm}\label{thm:r-bound}
  Let $n,t\in\NN^+$ such that $n\geq 6$, and
  let~$r$ be the smallest integer such that
  every set of $t$~binary trees on the leaf set~$[n]$ is displayed by
  some binary network with $r$~reticulations.
  Then
  \begin{equation*}
    r > \frac{(t-1)n \lg_2 n - 6tn - t\lg_2 t}{\lg_2 n + t + \lg_2 t} \in (t-1)n - O\left(\frac{nt^2}{\lg n + t}\right).
  \end{equation*}
\end{thm}

\begin{proof}
  Since every set of $t$~trees on the leaf set~$[n]$
  is displayed by a network with $(t - 1)n$~reticulations (see \cref{sec:intro}), we have $r \le (t - 1)n$.
  By the choice of $r$, each tree set $\T \in \S_{n,t}$ has a network~$N$ with $r$ reticulations
  such that $(N, \T) \in \mathcal{M}_{n,t,r}$.
  Therefore, $|\S_{n,t}|\leq |\mathcal{M}_{n,t,r}|$, so
  \begin{align*}
    \left(\frac{n}{2}\right)^{tn} \cdot t^{-t}
    \stackrel{\text{Obs.~\ref{obs:num-trees}}}{<} |\S_{n,t}|
    \leq |\mathcal{M}_{n,t,r}|
    \stackrel{\text{Cor.~\ref{cor:num-net-tree-pairs}}}{\leq} 2^{rt}\cdot n!\cdot (r - 1)!\cdot 2^{2n+6r-3}.
  \end{align*}
  By taking the base-2 logarithms on both sides, we obtain that
  \begin{align*}
    tn \lg n - tn - t\lg t 
    & < rt + \lg n! + \lg (r - 1)! + 2n + 6r - 3\\
    & \stackrel{\mathclap{\text{Lem.~\ref{lem:fac-bound}}}}{<} rt + n\lg n + r\lg r + n + 5r - 3\\
    & = r(t + \lg r + 5) + n(\lg n + 1) - 3\\
    &\stackrel{\mathclap{\vrule width 0pt depth 3pt n+r\leq tn}}{<} r(t + \lg tn) + n(\lg n + 5t).
  \end{align*}
  Thus,
  \begin{align*}
    r 
    & > \frac{tn\lg n - tn - t\lg t - n\lg n - 5tn}{\lg n + t + \lg t}
    = \frac{(t-1)n\lg n - 6tn - t\lg t}{\lg n + t + \lg t}\\
    & =(t-1)n - \frac{(t-1)n(t+\lg t) + 6tn + t\lg t}{\lg n + t + \lg t}
    \in (t-1)n - O\left(\frac{nt^2}{\lg n + t}\right).
    \qedhere
  \end{align*}
\end{proof}

\begin{cor}\label{cor:t-const}
  For all $t\in o\left(\sqrt{\lg n}\right)$, there is a set $\T$ of $t$ binary trees
  on the leaf set $[n]$ such that every binary network~$N$ that displays the trees in
  $\T$ has $(t-1)n - o(n)$ reticulations.
\end{cor}

\begin{cor}\label{cor:t-below-log}
  For all $t\in o(\lg n)$, there is a set $\T$ of $t$~binary trees
  on the leaf set $[n]$ such that every binary network $N$ that displays the trees in
  $\T$ has $(t-1)n - o(tn)$ reticulations.
\end{cor}

\begin{cor}\label{cor:t-log}
  Let~$t = c\lg_2 n$, for some constant~$c>0$.
  Then there is a set $\T$ of $t$~binary trees on the leaf set~$[n]$ such
  that every binary network~$N$ that displays the trees in $\T$ has
  \begin{equation*}
    r > \frac{(c\lg_2 n - 1)n \lg_2 n - O(n\lg n)}{(c+1)\lg_2 n + o(\lg n)} 
    \in \frac{c}{c+1+o(1)}n\lg_2 n - O(n)
  \end{equation*}
  reticulations.  In particular, as $c$ increases, $r$ approaches $n\lg_2 n$.
\end{cor}

\section{Counting Unrooted Binary Trees and Networks}

The argument in the previous section applies to rooted trees but generalizes to unrooted trees fairly naturally.
In the unrooted case, there is no concept of reticulations, but we can define the reticulation number
of a network~$\unrooted{N} = (V,E)$ as~$r(\unrooted{N}) := |E| - |V| + 1$, that is,
as the number of edges that need to be deleted from~$\unrooted{N}$ to obtain an unrooted tree.
This mirrors the rooted case in the sense that any rooting of an unrooted network~$\unrooted{N}$ has $r(\unrooted{N})$~reticulations.

The proof for rooted networks/trees had three main parts:
\begin{enumerate*}[(a)]
  \item bounding the number of networks~$N$ with reticulation number~$r$,
  \item bounding the number of pairs~$(N,\T)$ such that $\T$ is a set of $t$ trees on $n$ leaves and $N$ is a network with $r$ reticulations that displays the trees in $\T$, and
  \item bounding the number of sets of $t$ trees on the leaf set $[n]$.
\end{enumerate*}
Since there are $(2n - 5)!!$ unrooted trees with leaf set $[n]$, the number of
sets of $t$ unrooted trees on the leaf set $[n]$ is easy to bound:

\begin{obs}\label{obs:u-num-trees}
  Let~$n,t\in\NN^+$ with $n\geq 6$, and
  let~$\unrooted{\S}_{n,t}$ be the set of all
  sets of $t$ unrooted binary trees on the leaf set~$[n]$.
  Then
  \begin{equation*}
    |\unrooted{\S}_{n,t}|
    = \binom{(2n-5)!!}{t}
    \geq \left(\frac{(2n - 5)!!}{t}\right)^t
    \stackrel{\text{Lem.~\ref{lem:double-fac-bound}}}{>} \left(\frac{n - 1}{2}\right)^{t(n - 1)}\cdot t^{-t}.
  \end{equation*}
\end{obs}

Next, we bound the number of unrooted networks with $n$ leaves and reticulation
number~$r$. Such a bound is harder to obtain than for rooted networks, as the
definition of a (binary) unrooted network given in the preliminaries includes
all $3$-regular graphs, which clearly do not display any trees.  More
generally, a network may have pendant $2$-edge-connected components without
adjacent leaves.  These components have no bearing on the set of trees
displayed by the network but make it impossible to transfer one of the key
ideas in the proof of \cref{prop:num-networks} to unrooted networks: that the
only isomorphism between two reticulation-labelled networks is the identity.
For the purpose of this paper, we can restrict our attention to a subclass of
unrooted networks, which we call \emph{leaf-connecting}. A leaf-connecting
unrooted network has the property that every edge lies on a path between two
leaves.  It is easy to see that, if a set of unrooted trees $\unrooted{\T}$ is
displayed by an unrooted network $\unrooted{N}$, then the subnetwork of
$\unrooted{N}$ consisting of all edges contained in the embedding of at least
one tree in $\unrooted{\T}$ displays the trees in $\unrooted{\T}$, is
leaf-connecting, and has at most the same reticulation number as
$\unrooted{N}$.  Thus, a set of unrooted trees is displayed by an unrooted
network with reticulation number $r$ if and only if it is displayed by a
leaf-connecting unrooted network with reticulation number at most~$r$.

\begin{prop}\label{prop:u-num-networks}
  Let~$n,r\in\NN^+$, and
  let~$\unrooted{\N}_{n,r}$ be the set of all leaf-connecting unrooted
  networks with $n$~leaves and reticulation number~$r$.
  Then $|\unrooted{\N}_{n,r}| \le \frac{(2n + 4r - 5)!!}{r!} \leq n! \cdot (r - 2)! \cdot 2^{2n+6r-6}$.
\end{prop}

\begin{proof}
  The proof follows the same approach as in the rooted case, with adjustments
  to deal with the absence of edge directions.

  For an unrooted network $\unrooted{N} \in \unrooted{\N}_{n,r}$,
  a \emph{switching} of $\unrooted{N}$ is a labelling $\sigma : E(\unrooted{N}) \to \{0, \bot\}$
  such that the labelled edges form a spanning tree of $\unrooted{N}$. This means that
  there are $r$ unlabelled edges.  As in the rooted case, a~\emph{reticulation
  labelling} $\lambda_E$ of $\unrooted{N}$ is an extension of a switching
  $\sigma$ of $\unrooted{N}$ that bijectively labels all edges left
  unlabelled by $\sigma$ with labels in $[r]$, and a reticulation-labelled
  network is a pair $\uNLEnet$, where $\lambda_E$ is a
  reticulation labelling.
  Also as in the rooted case,
  we fix a particular switching~$\sigma_{\unrooted{N}}$ of $\unrooted{N}$,
  for each~$\unrooted{N} \in \unrooted{\N}_{n,r}$,
  we define~$\unrooted{\T}_m$ to be the set of unrooted trees with $m$~leaves, and
  we define the sets
  \begin{align*}
    \unetlabnr
    & := \{\uNLEnet \mid
      \unrooted{N} \in \unrooted{\N}_{n,r}
      \land \lambda_E \text{ is a reticulation labelling of $\unrooted{N}$}
    \land \sigma_{\unrooted{N}} \subseteq \lambda_E\} \text{ and}\\
    \utuplabnr
    & := \{\uNLEtup \mid
      \unrooted{N} \in \unrooted{\N}_{n,r}
      \land \lambda_E \text{ is a reticulation labelling of $\unrooted{N}$}
    \land \sigma_{\unrooted{N}}\subseteq \lambda_E\}.
  \end{align*}
  Analogously to the proof of \cref{prop:num-networks}, we show that
  \[
    |\unrooted{\N}_{n,r}| \cdot r!
    = |\utuplabnr|
    \stackrel{\text{(a)}}{=} |\unetlabnr|
    \stackrel{\text{(b)}}{\le} |\unrooted{\T}_{n + 2r}|.
  \]
  Again, it is obvious that $|\unrooted{\N}_{n,r}| \cdot r! =
  |\utuplabnr|$ because every switching
  $\sigma_{\unrooted{N}}$ has $r!$ extensions that are reticulation
  labellings.

  \medskip\noindent (b):
  The proof that $|\unetlabnr| \le |\unrooted{\T}_{n + 2r}|$
  uses essentially the same
  mapping~$\tau: \unetlabnr \to \unrooted{\T}_{n + 2r}$ as in the
  rooted case, replacing every edge $e = \{u, v\} \in E(\unrooted{N})$ with label
  $\lambda_E(e) > 0$ with two edges $\{u, z\}$ and $\{v, z'\}$, where $z$ and
  $z'$ are two new leaves labelled $n + 2\lambda_E(e) - 1$ and $n +
  2\lambda_E(e)$, respectively. Since it is once again easy to verify that each
  reticulation-labelled network $\uNLEnet$ can be recovered from
  its image $\tau(\uNLEnet)$, this mapping $\tau$ is injective,
  and we have $|\unetlabnr| \le |\unrooted{\T}_{n + 2r}|$.

  \medskip\noindent (a):
  To show that $|\unetlabnr| = |\utuplabnr|$, we
  prove that the existence of an isomorphism~$\phi: \NLEnet \to
  \NLEnet[\tilde]$ between two reticulation-labelled unrooted
  networks~$\NLEnet$ and $\NLEnet[\tilde]$ in $\unetlabnr$ implies
  that $\NLEtup = \NLEtup[\tilde]$. As in the rooted case, $\phi$ is also an
  isomorphism from $N$ to $\tilde N$, so $N = \tilde N$, but proving that this
  implies that $\phi = 1_N$ requires more care than in the rooted case.

  Since $\phi$ is an isomorphism from~$N$ to itself, it respects the labels of
  all leaves, so~$\phi(x) = x$ for all~$x \in \leaves(N)$.
  Now, assume for the sake of contradiction that $\phi \ne 1_N$.
  As there are no parallel edges in $N$ and $N$ is leaf-connecting,
  there is a path~$P$ between two leaves~$x$ and~$y$ in~$N$ that contains an
  edge~$\{u, v\}$ with~$\phi(u) = u$ and~$v' = \phi(v) \ne v$.
  Note that this implies that~$\phi(\{u, v\}) = \{u, v'\}$ and that
  neither~$v$ nor~$v'$ is a leaf of~$N$.

  Without loss of generality, suppose that~$v$ is on the subpath~$Q$ of~$P$ from $u$ to $y$.
  Then $Q' := \phi(Q)$ is also a path from $u = \phi(u)$ to $y = \phi(y)$.
  Note that, for every edge~$e \in Q$
  with $\lambda_E(e) \ne 0$, we have $\phi(e) = e$ because~$e$
  is the only edge with label $\lambda_E(e)$ and $\phi$ respects edge labels.
  Thus, all edges in the symmetric difference~$Q \oplus Q'$ of the edge sets of~$Q$ and~$Q'$
  are labelled~$0$ by $\lambda_E$.
  Moreover, $Q \oplus Q'$ is non-empty because it contains the edges~$\{u, v\}$
  and~$\{u, v'\}$, and every node is easily seen to have even degree in $Q \oplus Q'$.
  As every node in $N$ has degree at most three, this implies that $Q \oplus Q'$ is
  a non-empty collection of simple cycles all of whose edges are labelled~$0$ by $\lambda_E$,
  contradicting the fact that the edges labelled~$0$ by $\lambda_E$ form a spanning tree of $N$.
  Having obtained the desired contradiction, we conclude that $\phi = 1_N$ and, therefore, that
  $\NLEtup = \NLEtup[\tilde]$.

  \medskip\noindent
  Having shown (a)~and~(b), we conclude that
  \begin{align*}
    |\unrooted{\N}_{n,r}| 
    & \stackrel{\text{(a),(b)}}{\le}\frac{|\unrooted{\T}_{n+2r}|}{r!} = \frac{(2n + 4r - 5)!!}{r!}\\
    &\stackrel{\text{Obs.~\ref{obs:double-fac-eq}}}{=} \frac{(2n + 4r - 4)!}{2^{n + 2r - 2}(n + 2r - 2)!r!}
    \stackrel{\text{Obs.~\ref{obs:binom-bound}}}{\le} \frac{2^{2n + 4r - 4}((n + 2r - 2)!)^2}{2^{n + 2r - 2}(n + 2r - 2)!r!}
    = \frac{2^{n + 2r - 2}(n + 2r - 2)!}{r!}\\
    &\stackrel{\text{Obs.~\ref{obs:binom-bound}}}{\le} \frac{2^{2n + 4r - 4}n!(2r - 2)!}{r!}
    \stackrel{\text{Obs.~\ref{obs:binom-bound}}}{\le} \frac{2^{2n + 6r - 6}n!r!(r - 2)!}{r!}
    = 2^{2n + 6r - 6}n!(r - 2)!.\qedhere
  \end{align*}
\end{proof}

In addition to the bound on the number of networks with reticulation number~$r$
established by \cref{prop:u-num-networks}, we need the following bound on the
number of trees displayed by each of them:

\begin{prop}\label{prop:u-num-trees-per-net}
  Let~$\unrooted{N}\in\unrooted{\N}_{n,r}$ be a binary unrooted network with $n$~leaves and reticulation number~$r$.
  Then $\unrooted{N}$ displays at most $\binom{n + 3r - 3}{r} \le 2^{n + 3r - 3}$ binary unrooted trees with $n$~leaves.
\end{prop}
\begin{proof}
  For every tree~$\unrooted{T}$ displayed by~$\unrooted{N}$, there is a spanning tree of $\unrooted{N}$ that has a
  subdivision of $\unrooted{T}$ as a subtree.
  Conversely, for every spanning tree~$\unrooted{S}$ of~$\unrooted{N}$,
  there is exactly one tree displayed by~$\unrooted{N}$ such that $\unrooted{S}$ has a subdivision
  of this tree as a subtree.
  Thus, the number of unrooted trees displayed by~$\unrooted{N}$ is
  at most the number of spanning trees of $\unrooted{N}$.
  We show that this number is at most $2^{n + 3r - 3}$.
  
  As $\unrooted{N}$ is an unrooted binary network with $n$ leaves and reticulation number~$r$,
  we know that~$\unrooted{N}$ has exactly $2n + 3r - 3$~edges, and every spanning tree of $\unrooted{N}$ is
  obtained by deleting exactly~$r$ of them.
  Further, exactly $n$ edges of $\unrooted{N}$ are incident to leaves and are, therefore,
  in every spanning tree of $\unrooted{N}$.
  Thus, $\unrooted{N}$ has at most $\binom{n + 3r - 3}{r} \le 2^{n + 3r - 3}$ spanning trees.
\end{proof}

\Cref{prop:u-num-trees-per-net} immediately implies that every unrooted network
$\unrooted{N}$ displays at most $\binom{2^{n+3r-3}}{t} \le 2^{t(n+3r-3)}$ sets
of $t$ unrooted trees with leaf set $[n]$. Together with
\cref{prop:u-num-networks}, this proves the following corollary:

\begin{cor}\label{cor:u-num-net-tree-pairs}
  Let~$n,t,r\in\NN^+$ and
  let~$\unrooted{\M}_{n,t,r}$ be the set of all pairs~$(\unrooted{N},\unrooted{\T})$ such that
  $\unrooted{\T}$~is a set of $t$~binary unrooted trees on $n$~leaves and
  $\unrooted{N}$~is a leaf-connecting binary unrooted network with reticulation number~$r$
  that displays all unrooted trees in~$\unrooted{\T}$.
  Then $|\unrooted{\M}_{n,t,r}|\leq 2^{(t + 2)(n + 3r - 3)}\cdot n! \cdot (r - 2)!$.
\end{cor}

We can show the following analogue of \Cref{thm:r-bound} for unrooted trees now:

\begin{thm}\label{thm:u-r-bound}
  Let~$n,t\in\NN^+$ such that~$n\geq 6$, and
  let~$r$ be the smallest integer such that
  every set of $t$~binary unrooted trees on the leaf set~$[n]$
  is displayed by some binary unrooted network with $r$~reticulations.
  Then
  \begin{equation*}
    r > \frac{(t-1)n \lg_2 n - t(8n + \lg_2 n + \lg_2 t - 1)}{\lg_2 n + 3t + \lg_2 t}\\
    \in (t-1)n - O\left(\frac{nt^2}{\lg n + t}\right).
  \end{equation*}
\end{thm}

\begin{proof}
  Similar to the rooted case, we have
  \begin{equation*}
    \left(\frac{n - 1}{2}\right)^{t(n - 1)}\cdot t^{-t}
    \stackrel{\text{Obs.~\ref{obs:u-num-trees}}}{<} \binom{(2n-5)!!}{t} 
    = |\S^\text{u}_{n,t}| 
    \leq |\mathcal{M}^\text{u}_{n,t,r}|
    \stackrel{\text{Cor.~\ref{cor:u-num-net-tree-pairs}}}{\leq} 2^{(t + 2)(n + 3r - 3)}\cdot n!\cdot (r - 2)!,
  \end{equation*}
  and we can assume that $n + r \le tn$.
  Taking base-2 logarithms again gives
  \begin{align*}
    t(n - 1)\lg (n - 1) - t(n - 1) - t \lg t
    & < (t + 2)(n + 3r - 3) + \lg n! + \lg (r - 2)!\\
    & \stackrel{\mathclap{\text{\vrule width 0pt depth 3pt Lem.~\ref{lem:fac-bound}}}}{<} (t + 1)n + (t + 2)(3r - 3) + n\lg n + r\lg r\\
    & = r(3t + \lg r + 6) + n(t + \lg n + 1) - 3t - 6\\
    & \stackrel{\mathclap{\vrule width 0pt depth 3pt n+r\leq tn}}{<} r(3t + \lg tn) + n(\lg n + 7t)
  \end{align*}
  and, thus,
  \begin{align*}
    r
    & > \frac{t(n - 1)\lg n - t(n - 1) - t\lg t - n\lg n - 7tn}{\lg n + 3t + \lg t}
    = \frac{(t-1)n\lg n - t(8n + \lg n + \lg t - 1)}{\lg n + 3t + \lg t}\\
    & =(t-1)n - \frac{(t-1)n(3t+\lg t) + t(8n + \lg n + \lg t - 1)}{\lg n + 3t + \lg t}
    \in (t-1)n - O\left(\frac{nt^2}{\lg n + t}\right).
    \qedhere
  \end{align*}
\end{proof}

A case analysis similar to the rooted case now proves the following corollaries.

\begin{cor}\label{cor:sub-root-log-trees-unrooted}
  For all $t \in o\left(\sqrt{\lg n}\right)$, there is a set $\unrooted{\T}$ of $t$~binary unrooted
  trees on the leaf set $[n]$ such that every binary unrooted network $\unrooted{N}$
  that displays the trees in $\unrooted{\T}$ has reticulation number at least~$(t - 1)n - o(n)$.
\end{cor}

\begin{cor}\label{cor:sublog-trees-unrooted}
  For all $t \in o(\lg n)$, there is a set $\unrooted{\T}$ of $t$~binary unrooted trees on
  the leaf set $[n]$ such that every binary unrooted network $\unrooted{N}$ that
  displays the trees in $\unrooted{\T}$ has reticulation number at least~$(t - 1)n - o(tn)$.
\end{cor}

\begin{cor}\label{cor:log-trees-unrooted}
  Let~$t = c\lg_2 n$, for some constant~$c>0$.
  Then there is a set $\unrooted{\T}$ of $t$ binary unrooted trees on the leaf set
  $[n]$ such that every binary unrooted network $\unrooted{N}$ that displays the trees in
  $\unrooted{\T}$ has reticulation number
  \begin{equation*}
    r > \frac{(c\lg_2 n - 1)n \lg_2 n - O(n\lg n)}{(3c+1)\lg_2 n + o(\lg n)} 
    \in \frac{c}{3c+1+o(1)}n\lg_2 n - O(n).
  \end{equation*}
  In particular, as $c$ increases, $r$ approaches $\frac{1}{3}n\lg_2 n$.
\end{cor}

\section{Conclusion}

We have shown that, for any sub-logarithmic (in the number of leaves) number~$t$,
there is some set of $t$~binary trees for which the trivial binary network
with~$(t - 1)n$~reticulations is the best possible network that displays these trees, up to lower-order terms.
This implies that these trees have almost no common structure.
In particular, this result holds for any constant number of trees, which,
as discussed in the introduction,
has the important implication that cluster reduction is not safe
for computing the hybridization number of four or more trees.
This can be interpreted as a strong argument against phylogenetic
reconstruction by parsimony as follows.
All input trees admitting the same cluster can
reasonably be considered a strong biological signal.
Our results imply that no network displaying all trees can represent this
signal if it is forced to have a minimum reticulation number.
Indeed, it is more reasonable to expect the correct evolutionary history
to reside among slightly non-optimal solutions.

Our results also imply that, for any $c>0$, there are sets of $c\lg_2 n$~binary
trees that cannot be displayed by a binary network with $o(n \lg n)$~reticulations.
This suggests that most of the
reticulations in the optimal network that displays \emph{all} trees on $n$~leaves,
which has $\Theta(n \lg n)$~reticulations~\cite{BS18},
are due to only a tiny fraction of the exponentially many trees.

Two obvious open questions remain: First, for $t\in o(\lg n)$, we believe that
there are sets of $t$~binary trees that require~$(t - 1)n - o(n)$ reticulations to
display, even though we were only able to prove the weaker bound
of $(t-1)n-o(tn)$.
This is probably due to substantial overcounting of the number of sets of
$t$~trees that can be displayed by some network with $r$~reticulations.
In particular, we bound the number of these sets by the number of
pairs~$(N,\T)$ such that $\T$ is a set of $t$~trees, and~$N$ is a network
with $r$~reticulations that displays~$\T$.
We suspect that many such sets~$\T$ are displayed by many networks with
$r$~reticulations, so we count them often. There are other sources of
overcounting in our proofs, but this is most likely the most egregious.

The second open question concerns the bound on the number of reticulations
needed to display $c\lg_2 n$~unrooted binary trees in
\cref{cor:log-trees-unrooted}.
This bound is roughly one third of the bound for rooted trees shown in
\cref{cor:t-log}.
We believe that this gap is a caveat of our proof and that
there is also a set of $c\lg_2 n$ unrooted trees that require
close to $n \lg_2 n$~reticulations to display as $c$ increases.

\printbibliography[notcategory=ignore]

\end{document}